\documentclass[12pt,fleqn]{article}

\usepackage{amsmath,amssymb,amsthm}
\usepackage{geometry}
\usepackage[pdfpagemode=UseNone,pdfstartview=FitH]{hyperref}

\newcommand{\bdry}[1]{\partial #1}
\newcommand{\bgset}[1]{\big\{#1\big\}}
\newcommand{\B}{{\cal B}}
\newcommand{\closure}[1]{\overline{#1}}
\newcommand{\dint}{\ds{\int}}
\newcommand{\ds}[1]{\displaystyle #1}
\newcommand{\eps}{\varepsilon}
\newcommand{\goodchi}{\protect\raisebox{2pt}{$\chi$}}
\newcommand{\half}{\frac{1}{2}}
\newcommand{\interior}[1]{#1^\circ}
\newcommand{\loc}{\text{loc}}
\renewcommand{\L}{{\mathcal L}}
\newcommand{\norm}[2][]{\left\|#2\right\|_{#1}}
\newcommand{\PS}[1]{$(\text{PS})_{#1}$}
\newcommand{\QED}{\mbox{\qedhere}}
\newcommand{\R}{\mathbb R}
\newcommand{\seq}[1]{\left(#1\right)}
\newcommand{\set}[1]{\left\{#1\right\}}
\newcommand{\strictsubset}{\subset \subset}
\newcommand{\vol}[1]{\if #1'' \L \else \L(#1) \fi}

\DeclareMathOperator{\divg}{div}
\DeclareMathOperator{\supp}{supp}

\newenvironment{enumroman}{\begin{enumerate}

}{\end{enumerate}}

\newtheorem{lemma}{Lemma}[section]
\newtheorem{theorem}[lemma]{Theorem}

\numberwithin{equation}{section} 

\title{\bf On a class of elliptic free boundary problems with multiple solutions}
\author{\bf Kanishka Perera\\
Department of Mathematical Sciences\\
Florida Institute of Technology\\
Melbourne, FL 32901, USA\\
\em kperera@fit.edu}
\date{}

\begin{document}

\maketitle

\begin{abstract}
We prove that a certain class of elliptic free boundary problems, which includes the Prandtl-Batchelor problem from fluid dynamics as a special case, has two distinct nontrivial solutions for large values of a parameter. The first solution is a global minimizer of the energy. The energy functional is nondifferentiable, so standard variational arguments cannot be used directly to obtain a second nontrivial solution. We obtain our second solution as the limit of mountain pass points of a sequence of $C^1$-functionals approximating the energy. We use careful estimates of the corresponding energy levels to show that this limit is neither trivial nor a minimizer.
\end{abstract}

\newpage

\section{Introduction}

Consider the class of sublinear elliptic free boundary problems
\begin{equation} \label{1.1}
\left\{\begin{aligned}
- \Delta u & = \lambda\, \goodchi_{\set{u > 1}}(x)\, g(x,(u - 1)_+) && \text{in } \Omega \setminus F(u)\\[10pt]
|\nabla u^+|^2 - |\nabla u^-|^2 & = 2 && \text{on } F(u)\\[10pt]
u & = 0 && \text{on } \bdry{\Omega},
\end{aligned}\right.
\end{equation}
where $\Omega$ is a bounded domain in $\R^N,\, N \ge 2$ with $C^{2,\alpha}$-boundary $\bdry{\Omega}$,
\[
F(u) = \bdry{\set{u > 1}}
\]
is the free boundary of $u$, $\lambda > 0$ is a parameter, $\goodchi_{\set{u > 1}}$ is the characteristic function of the set $\set{u > 1}$, $(u - 1)_+ = \max\, (u - 1,0)$ is the positive part of $u - 1$, $\nabla u^\pm$ are the limits of $\nabla u$ from the sets $\set{u > 1}$ and $\interior{\set{u \le 1}}$, respectively, and $g : \Omega \times [0,\infty) \to [0,\infty)$ is a locally H\"{o}lder continuous function satisfying
\begin{enumerate}
\item[$(g_1)$] for some $a_1, a_2 > 0$ and $1 < p < 2$,
    \[
    |g(x,s)| \le a_1 + a_2\, s^{p-1} \quad \forall (x,s) \in \Omega \times [0,\infty),
    \]
\item[$(g_2)$] $g(x,s) > 0$ for all $x \in \Omega$ and $s > 0$.
\end{enumerate}
The purpose of this paper is to prove that this problem has two distinct nontrivial (suitably generalized) solutions for all sufficiently large $\lambda$.

The special case $g(x,s) \equiv 1$ is the well-known Prandtl-Batchelor free boundary problem, where the phase $\set{u > 1}$ represents a vortex patch bounded by the vortex line $u = 1$ in a steady-state fluid flow when $N = 2$ (see Batchelor \cite{MR0084296,MR0084310}). This particular case has been studied in Caflisch \cite{MR1001785}, Elcrat and Miller \cite{MR1285987}, Acker \cite{MR1606934,MR1910339}, and Jerison and Perera \cite{JePe}. Problem \eqref{1.1} also arises in the confinement of a plasma by a magnetic field, where the region $\set{u > 1}$ represents the plasma and the boundary of the plasma is the free boundary (see, e.g., Temam \cite{MR0412637, MR0602544}, Caffarelli and Friedman \cite{MR587175}, Friedman and Liu \cite{MR1360544}, and Jerison and Perera \cite{MR3790500}).

The solutions of problem \eqref{1.1} that we construct here are Lipschitz continuous functions of class $H^1_0(\Omega) \cap C^2(\closure{\Omega} \setminus F(u))$ that satisfy the equation $- \Delta u = \lambda\, \goodchi_{\set{u > 1}}(x)\, g(x,(u - 1)_+)$ in the classical sense in $\Omega \setminus F(u)$ and vanish continuously on $\bdry{\Omega}$. They satisfy the free boundary condition in the following generalized sense: for all $\Phi \in C^1_0(\Omega,\R^N)$ such that $u \ne 1$ a.e.\! on the support of $\Phi$,
\[
\lim_{\delta^+ \searrow 0}\, \int_{\set{u = 1 + \delta^+}} \left(2 - |\nabla u|^2\right) \Phi \cdot n\, d\sigma - \lim_{\delta^- \searrow 0}\, \int_{\set{u = 1 - \delta^-}} |\nabla u|^2\, \Phi \cdot n\, d\sigma = 0,
\]
where $n$ is the outward unit normal to $\set{1 - \delta^- < u < 1 + \delta^+}$ (the sets $\set{u = 1 \pm \delta^\pm}$ are smooth hypersurfaces for a.a.\! $\delta^\pm > 0$ by Sard's theorem and the above limits are taken through such $\delta^\pm$). In particular, the free boundary condition is satisfied in the classical sense on any smooth portion of $F(u)$.

The variational functional associated with problem \eqref{1.1} is given by
\[
J(u) = \int_\Omega \left[\half\, |\nabla u|^2 + \goodchi_{\set{u > 1}}(x) - \lambda\, G(x,(u - 1)_+)\right] dx, \quad u \in H^1_0(\Omega),
\]
where
\[
G(x,s) = \int_0^s g(x,t)\, dt, \quad s \ge 0.
\]
We will prove the following multiplicity result.

\begin{theorem} \label{Theorem 1.1}
Assume $(g_1)$ and $(g_2)$. Then there exists a $\lambda^\ast > 0$ such that for all $\lambda > \lambda^\ast$, problem \eqref{1.1} has two Lipschitz continuous solutions $u_0, u_1 \in H^1_0(\Omega) \cap C^2(\closure{\Omega} \setminus F(u))$ that satisfy the equation $- \Delta u = \lambda\, \goodchi_{\set{u > 1}}(x)\, g(x,(u - 1)_+)$ in the classical sense in $\Omega \setminus F(u)$, the free boundary condition in the generalized sense, and vanish continuously on $\bdry{\Omega}$. Moreover,
\begin{enumroman}
\item $J(u_0) < - \vol{\Omega} \le - \vol{{\set{u_1 = 1}}} < J(u_1)$, where $\vol{}$ denotes the Lebesgue measure in $\R^N$, and hence $u_0$ and $u_1$ are nontrivial and distinct;
\item $0 < u_1 \le u_0$, the sets $\set{u_0 < 1} \subset \set{u_1 < 1}$ are connected if $\bdry{\Omega}$ is connected, and the sets $\set{u_0 > 1} \supset \set{u_1 > 1}$ are nonempty;
\item $u_0$ is a minimizer of $J$, but $u_1$ is not a minimizer of $J$.
\end{enumroman}
\end{theorem}

This theorem will be proved in the next section. Since $u_0$ is a minimizer of $J$, it follows from standard arguments that it satisfies the free boundary condition in the viscosity sense and its free boundary $F(u_0)$ has finite $(N - 1)$-dimensional Hausdorff measure and is a smooth hypersurface except on a closed set of Hausdorff dimension at most $N - 3$. Near the smooth subset of $F(u_0)$, $(u_0 - 1)_\pm$ are smooth and the free boundary condition is satisfied in the classical sense (see, e.g., Caffarelli and Salsa \cite{MR2145284}). The nondegeneracy and regularity of $u_1$ is presently an open problem.

\section{Proof of Theorem \ref{Theorem 1.1}}

Since the functional $J$ is nondifferentiable, we approximate it by $C^1$-functionals as follows. Let $\beta : \R \to [0,2]$ be a smooth function such that $\beta(s) = 0$ for $s \le 0$, $\beta(s) > 0$ for $0 < s < 1$, $\beta(s) = 0$ for $s \ge 1$, and $\int_0^1 \beta(s)\, ds = 1$. Then let
\[
\B(s) = \int_0^s \beta(t)\, dt
\]
and note that $\B : \R \to [0,1]$ is a smooth nondecreasing function such that $\B(s) = 0$ for $s \le 0$, $0 < \B(s) < 1$ for $0 < s < 1$, and $\B(s) = 1$ for $s \ge 1$. For $\eps > 0$, let
\[
g_\eps(x,s) = \B\left(\frac{s}{\eps}\right) g(x,s), \quad G_\eps(x,s) = \int_0^s g_\eps(x,t)\, dt, \quad s \ge 0
\]
and set
\[
J_\eps(u) = \int_\Omega \left[\half\, |\nabla u|^2 + \B\left(\frac{u - 1}{\eps}\right) - \lambda\, G_\eps(x,(u - 1)_+)\right] dx, \quad u \in H^1_0(\Omega).
\]
The functional $J_\eps$ is of class $C^1$ and its critical points coincide with weak solutions of the problem
\begin{equation} \label{2.1}
\left\{\begin{aligned}
- \Delta u & = - \frac{1}{\eps}\, \beta\left(\frac{u - 1}{\eps}\right) + \lambda\, g_\eps(x,(u - 1)_+) && \text{in } \Omega\\[10pt]
u & = 0 && \text{on } \bdry{\Omega}.
\end{aligned}\right.
\end{equation}
If $u \in H^1_0(\Omega)$ is a weak solution of this problem, then $u \in C^{2,\alpha}(\closure{\Omega})$ and is a classical solution by elliptic regularity theory. If $u$ is not identically zero, then it is nontrivial in a stronger sense, namely, $u > 0$ in $\Omega$ and $u > 1$ in a nonempty open set. Indeed, if $u \le 1$ everywhere, then $u$ is harmonic in $\Omega$ and hence vanishes identically since $u = 0$ on $\bdry{\Omega}$. Furthermore, in the set $\set{u < 1}$, $u$ is the harmonic function with boundary values $0$ on $\bdry{\Omega}$ and $1$ on $\bdry{\set{u \ge 1}}$, and hence strictly positive since $\Omega$ is connected.

First we prove the following convergence result.

\begin{lemma} \label{Lemma 2.1}
Assume $(g_1)$ and $(g_2)$. Let $\eps_j \searrow 0$ and let $u_j$ be a critical point of $J_{\eps_j}$. If the sequence $\seq{u_j}$ is bounded in $H^1_0(\Omega) \cap L^\infty(\Omega)$, then there exists a Lipschitz continuous function $u$ on $\closure{\Omega}$ such that $u \in H^1_0(\Omega) \cap C^2(\closure{\Omega} \setminus F(u))$ and, for a renamed subsequence,
\begin{enumroman}
\item \label{Lemma 2.1.i} $u_j \to u$ uniformly on $\closure{\Omega}$,
\item \label{Lemma 2.1.ii} $u_j \to u$ locally in $C^1(\closure{\Omega} \setminus \set{u = 1})$,
\item \label{Lemma 2.1.iii} $u_j \to u$ strongly in $H^1_0(\Omega)$,
\item \label{Lemma 2.1.iv} $J(u) \le \liminf J_{\eps_j}(u_j) \le \limsup J_{\eps_j}(u_j) \le J(u) + \vol{{\set{u = 1}}}$, in particular, $u$ is nontrivial if $\liminf J_{\eps_j}(u_j) < 0$ or $\limsup J_{\eps_j}(u_j) > 0$.
\end{enumroman}
Moreover, $u$ satisfies the equation $- \Delta u = \lambda\, \goodchi_{\set{u > 1}}(x)\, g(x,(u - 1)_+)$ in the classical sense in $\Omega \setminus F(u)$, the free boundary condition in the generalized sense, and vanishes continuously on $\bdry{\Omega}$. If $u$ is nontrivial, then $u > 0$ in $\Omega$, the set $\set{u < 1}$ is connected if $\bdry{\Omega}$ is connected, and the set $\set{u > 1}$ is nonempty.
\end{lemma}

The crucial ingredient in the passage to the limit in the proof of this lemma is the following uniform Lipschitz continuity result of Caffarelli et al.\! \cite{MR1906591}.

\begin{lemma}[{\cite[Theorem 5.1]{MR1906591}}] \label{Lemma 2.2}
Let $u$ be a Lipschitz continuous function on $B_1(0) \subset \R^N$ satisfying the distributional inequalities
\[
\pm \Delta u \le A \left(\frac{1}{\eps}\, \goodchi_{\set{|u - 1| < \eps}}(x) + 1\right)
\]
for some constants $A > 0$ and $0 < \eps \le 1$. Then there exists a constant $C > 0$, depending on $N$, $A$, and $\dint_{B_1(0)} u^2\, dx$, but not on $\eps$, such that
\[
\sup_{x \in B_{1/2}(0)}\, |\nabla u(x)| \le C.
\]
\end{lemma}

\begin{proof}[Proof of Lemma \ref{Lemma 2.1}]
We may assume that $0 < \eps_j \le 1$. The function $u_j$ is a solution of
\begin{equation} \label{2.2}
\left\{\begin{aligned}
- \Delta u_j & = - \frac{1}{\eps_j}\, \beta\left(\frac{u_j - 1}{\eps_j}\right) + \lambda\, g_{\eps_j}(x,(u_j - 1)_+) && \text{in } \Omega\\[10pt]
u_j & = 0 && \text{on } \bdry{\Omega}.
\end{aligned}\right.
\end{equation}
Since $\seq{u_j}$ is bounded in $L^\infty(\Omega)$, $0 \le g_{\eps_j}(x,(u_j - 1)_+) \le A_0$ for some constant $A_0 > 0$ by $(g_1)$. Let $\varphi_0 > 0$ be the solution of
\[
\left\{\begin{aligned}
- \Delta \varphi_0 & = \lambda A_0 && \text{in } \Omega\\[10pt]
\varphi_0 & = 0 && \text{on } \bdry{\Omega}.
\end{aligned}\right.
\]
Since $\beta \ge 0$, $- \Delta u_j \le \lambda A_0$ in $\Omega$, and hence
\[
0 \le u_j(x) \le \varphi_0(x) \quad \forall x \in \Omega
\]
by the maximum principle. The majorant $\varphi_0$ gives a uniform lower bound $\delta_0 > 0$ on the distance from the set $\set{u_j \ge 1}$ to $\bdry{\Omega}$. Since $u_j$ is positive, harmonic, and bounded by $1$ in a $\delta_0$-neighborhood of $\bdry{\Omega}$, it follows from standard boundary regularity theory that the sequence $\seq{u_j}$ is bounded in the $C^{2,\alpha}$ norm, and hence compact in the $C^2$ norm, in a $\delta_0/2$-neighborhood.

Since $0 \le \beta \le 2\, \goodchi_{(-1,1)}$,
\[
\pm \Delta u_j = \pm \frac{1}{\eps_j}\, \beta\left(\frac{u_j - 1}{\eps_j}\right) \mp \lambda\, g_{\eps_j}(x,(u_j - 1)_+) \le \frac{2}{\eps_j}\, \goodchi_{\set{|u_j - 1| < \eps_j}}(x) + \lambda A_0.
\]
Since $\seq{u_j}$ is bounded in $L^2(\Omega)$, it follows from this and Lemma \ref{Lemma 2.2} that there exists a constant $C > 0$ such that
\[
\max_{x \in B_{r/2}(x_0)}\, |\nabla u_j(x)| \le \frac{C}{r}
\]
whenever $r > 0$ and $B_r(x_0) \subset \Omega$. Hence $u_j$ is uniformly Lipschitz continuous on the compact subset of $\Omega$ at distance greater or equal to $\delta_0/2$ from $\bdry{\Omega}$.

Thus, a renamed subsequence of $\seq{u_j}$ converges uniformly on $\closure{\Omega}$ to a Lipschitz continuous function $u$ with zero boundary values, with strong convergence in $C^2$ on a $\delta_0/2$-neighborhood of $\bdry{\Omega}$. Since $\seq{u_j}$ is bounded in $H^1_0(\Omega)$, a further subsequence converges weakly in $H^1_0(\Omega)$ to $u$.

Next we show that $u$ satisfies the equation $- \Delta u = \lambda\, \goodchi_{\set{u > 1}}(x)\, g(x,(u - 1)_+)$ in the set $\set{u \ne 1}$. Let $\varphi \in C^\infty_0(\set{u > 1})$. Then $u \ge 1 + 2\, \eps$ on the support of $\varphi$ for some $\eps > 0$. For all sufficiently large $j$, $\eps_j < \eps$ and $|u_j - u| < \eps$ in $\Omega$, so $u_j \ge 1 + \eps_j$ on the support of $\varphi$. So testing \eqref{2.2} with $\varphi$ gives
\[
\int_\Omega \nabla u_j \cdot \nabla \varphi\, dx = \int_\Omega \lambda\, g(x,u_j - 1)\, \varphi\, dx,
\]
and passing to the limit gives
\[
\int_\Omega \nabla u \cdot \nabla \varphi\, dx = \int_\Omega \lambda\, g(x,u - 1)\, \varphi\, dx
\]
since $u_j$ converges to $u$ weakly in $H^1_0(\Omega)$ and uniformly on $\Omega$. Hence $u$ is a distributional, and hence a classical, solution of $- \Delta u = \lambda\, g(x,u - 1)$ in the set $\set{u > 1}$. A similar argument shows that $u$ satisfies $\Delta u = 0$ in the set $\set{u < 1}$.

Now we show that $u$ is also harmonic in the possibly larger set $\interior{\set{u \le 1}}$. Since $\beta \ge 0$ and $\B \le 1$, testing \eqref{2.2} with any nonnegative test function and passing to the limit shows that
\begin{equation} \label{2.3}
- \Delta u \le \lambda\, g(x,(u - 1)_+) \quad \text{in } \Omega
\end{equation}
in the distributional sense. On the other hand, since $u$ is harmonic in $\set{u < 1}$, $\mu := \Delta (u - 1)_-$ is a nonnegative Radon measure supported on $\Omega \cap \bdry{\set{u < 1}}$ by Alt and Caffarelli \cite[Remark 4.2]{MR618549}, so
\begin{equation} \label{2.4}
- \Delta u = \mu \ge 0 \quad \text{in } \set{u \le 1}.
\end{equation}
It follows from \eqref{2.3} and \eqref{2.4} that $u \in W^{2,\, p}_\loc(\interior{\set{u \le 1}}),\, 1 < p < \infty$ and hence $\mu$ is actually supported on $\Omega \cap \bdry{\set{u < 1}} \cap \bdry{\set{u > 1}}$, so $u$ is harmonic in $\interior{\set{u \le 1}}$.

Since $u_j$ converges in the $C^2$ norm to $u$ in a neighborhood of $\bdry{\Omega}$ in $\closure{\Omega}$, it suffices to show that $u_j \to u$ locally in $C^1(\Omega \setminus \set{u = 1})$ to prove \ref{Lemma 2.1.ii}. Let $U \strictsubset \set{u > 1}$. Then $u \ge 1 + 2\, \eps$ in $U$ for some $\eps > 0$. For all sufficiently large $j$, $\eps_j < \eps$ and $|u_j - u| < \eps$ in $\Omega$, so $u_j \ge 1 + \eps_j$ in $U$. So \eqref{2.2} gives $- \Delta u_j = \lambda\, g(x,u_j - 1)$ in $U$. Since $g$ is locally H\"{o}lder continuous and $u_j \to u$ uniformly, $g(x,u_j - 1) \to g(x,u - 1)$ in $L^p(U)$ for $1 < p < \infty$. Since $- \Delta u = \lambda\, g(x,u - 1)$ in $U$, then $u_j \to u$ in $W^{2,p}(U)$. Since $W^{2,p}(U) \hookrightarrow C^1(U)$ for $p > 2$, it follows that $u_j \to u$ in $C^1(U)$. A similar argument shows that $u_j \to u$ locally in $C^1(\set{u < 1})$ also.

Since $u_j \rightharpoonup u$ in $H^1_0(\Omega)$, $\norm{u} \le \liminf \norm{u_j}$, so it suffices to show that $\limsup \norm{u_j} \linebreak \le \norm{u}$ to prove \ref{Lemma 2.1.iii}. Multiplying the first equation in \eqref{2.2} by $u_j - 1$, integrating by parts, and noting that $\beta(s/\eps_j)\, s \ge 0$ for all $s$ gives
\begin{multline} \label{2.5}
\int_\Omega |\nabla u_j|^2\, dx \le \int_\Omega \lambda\, g(x,(u_j - 1)_+)\, (u_j - 1)_+\, dx - \int_{\bdry{\Omega}} \frac{\partial u_j}{\partial n}\, d\sigma\\[7.5pt]
\to \int_\Omega \lambda\, g(x,(u - 1)_+)\, (u - 1)_+\, dx - \int_{\bdry{\Omega}} \frac{\partial u}{\partial n}\, d\sigma,
\end{multline}
where $n$ is the outward unit normal to $\bdry{\Omega}$. Fix $0 < \eps < 1$. Recall that $u$ is a solution of $- \Delta u = \lambda\, g(x,u - 1)$ in $\set{u > 1}$. Testing this equation with $\varphi = (u - 1 - \eps)_+$ gives
\begin{equation} \label{2.6}
\int_{\set{u > 1 + \eps}} |\nabla u|^2\, dx = \int_\Omega \lambda\, g(x,(u - 1)_+)\, (u - 1 - \eps)_+\, dx.
\end{equation}
Integrating $(u - 1 + \eps)_-\, \Delta u = 0$ over $\Omega$ gives
\begin{equation} \label{2.7}
\int_{\set{u < 1 - \eps}} |\nabla u|^2\, dx = - (1 - \eps) \int_{\bdry{\Omega}} \frac{\partial u}{\partial n}\, d\sigma.
\end{equation}
Adding \eqref{2.6} and \eqref{2.7}, and letting $\eps \searrow 0$ gives
\[
\int_\Omega |\nabla u|^2\, dx = \int_\Omega \lambda\, g(x,(u - 1)_+)\, (u - 1)_+\, dx - \int_{\bdry{\Omega}} \frac{\partial u}{\partial n}\, d\sigma
\]
since $\dint_{\set{u = 1}} |\nabla u|^2\, dx = 0$. This together with \eqref{2.5} gives
\[
\limsup \int_\Omega |\nabla u_j|^2\, dx \le \int_\Omega |\nabla u|^2\, dx
\]
as desired.

To prove \ref{Lemma 2.1.iv}, write
\begin{multline*}
J_{\eps_j}(u_j) = \int_\Omega \left[\half\, |\nabla u_j|^2 + \B\left(\frac{u_j - 1}{\eps_j}\right) \goodchi_{\set{u \ne 1}}(x) - \lambda\, G_{\eps_j}(x,(u_j - 1)_+)\right] dx\\[7.5pt]
+ \int_{\set{u = 1}} \B\left(\frac{u_j - 1}{\eps_j}\right) dx.
\end{multline*}
Since $u_j \to u$ in $H^1_0(\Omega)$, and $\B((u_j - 1)/\eps_j)\, \goodchi_{\set{u \ne 1}}$ and $G_{\eps_j}(x,(u_j - 1)_+)$ are bounded and converge pointwise to $\goodchi_{\set{u > 1}}$ and $G(x,(u - 1)_+)$, respectively, the first integral converges to $J(u)$. Since
\[
0 \le \int_{\set{u = 1}} \B\left(\frac{u_j - 1}{\eps_j}\right) dx \le \vol{{\set{u = 1}}},
\]
the desired conclusion follows.

Finally we show that $u$ satisfies the free boundary condition in the generalized sense. Let $\Phi \in C^1_0(\Omega,\R^N)$ be such that $u \ne 1$ a.e.\! on the support of $\Phi$. Multiplying the first equation in \eqref{2.2} by $\nabla u_j \cdot \Phi$ and integrating over the set $\set{1 - \delta^- < u < 1 + \delta^+}$ gives
\begin{multline*}
\int_{\set{1 - \delta^- < u < 1 + \delta^+}} \left[- \Delta u_j + \frac{1}{\eps_j}\, \beta\left(\frac{u_j - 1}{\eps_j}\right)\right] \nabla u_j \cdot \Phi\, dx\\[7.5pt]
= \int_{\set{1 - \delta^- < u < 1 + \delta^+}} \lambda\, g_{\eps_j}(x,(u_j - 1)_+)\, \nabla u_j \cdot \Phi\, dx.
\end{multline*}
Noting that the integrand on the left-hand side is equal to
\[
\divg \left(\half\, |\nabla u_j|^2\, \Phi - (\nabla u_j \cdot \Phi)\, \nabla u_j\right) + \nabla u_j\, D\Phi \cdot \nabla u_j - \half\, |\nabla u_j|^2\, \divg \Phi + \nabla \B\left(\frac{u_j - 1}{\eps_j}\right) \cdot \Phi
\]
and integrating by parts gives
\begin{multline} \label{2.8}
\int_{\set{u = 1 + \delta^+} \cup \set{u = 1 - \delta^-}} \left[\half\, |\nabla u_j|^2\, \Phi - (\nabla u_j \cdot \Phi)\, \nabla u_j + \B\left(\frac{u_j - 1}{\eps_j}\right) \Phi\right] \cdot n\, d\sigma\\[7.5pt]
= \int_{\set{1 - \delta^- < u < 1 + \delta^+}} \left(\half\, |\nabla u_j|^2\, \divg \Phi - \nabla u_j\, D\Phi \cdot \nabla u_j\right) dx\\[7.5pt]
+ \int_{\set{1 - \delta^- < u < 1 + \delta^+}} \left[\B\left(\frac{u_j - 1}{\eps_j}\right) \divg \Phi + \lambda\, g_{\eps_j}(x,(u_j - 1)_+)\, \nabla u_j \cdot \Phi\right] dx.
\end{multline}
By \ref{Lemma 2.1.ii}, the integral on the left-hand side converges to
\[
\int_{\set{u = 1 + \delta^+} \cup \set{u = 1 - \delta^-}} \left(\half\, |\nabla u|^2\, \Phi - (\nabla u \cdot \Phi)\, \nabla u\right) \cdot n\, d\sigma + \int_{\set{u = 1 + \delta^+}} \Phi \cdot n\, d\sigma,
\]
which is equal to
\[
\int_{\set{u = 1 + \delta^+}} \left(1 - \half\, |\nabla u|^2\right) \Phi \cdot n\, d\sigma - \int_{\set{u = 1 - \delta^-}} \half\, |\nabla u|^2\, \Phi \cdot n\, d\sigma
\]
since $n = \pm \nabla u/|\nabla u|$ on $\set{u = 1 \pm \delta^\pm}$. The first integral on the right-hand side of \eqref{2.8} converges to
\[
\int_{\set{1 - \delta^- < u < 1 + \delta^+}} \left(\half\, |\nabla u|^2\, \divg \Phi - \nabla u\, D\Phi \cdot \nabla u\right) dx
\]
by \ref{Lemma 2.1.iii}, and the second integral is bounded by
\[
\int_{\set{1 - \delta^- < u < 1 + \delta^+}} \big(|\divg \Phi| + a_3\, |\Phi|\big)\, dx
\]
for some constant $a_3 > 0$. Since $\vol{{\set{u = 1} \cap \supp \Phi}} = 0$, the last two integrals go to zero as $\delta^\pm \searrow 0$. So first letting $j \to \infty$ and then letting $\delta^\pm \searrow 0$ in \eqref{2.8} gives the desired conclusion.
\end{proof}

By $(g_1)$,
\[
J_\eps(u) \ge \int_\Omega \left(\half\, |\nabla u|^2 - \lambda \left[a_1\, (u - 1)_+ + \frac{a_2}{p}\, (u - 1)_+^p\right]\right) dx,
\]
and since $1 < p < 2$, this implies that $J_\eps$ is bounded from below and coercive. Hence $J_\eps$ satisfies the \PS{} condition, i.e., every sequence $\seq{u_j} \subset H^1_0(\Omega)$ such that $J_\eps(u_j)$ is bounded and $J_\eps'(u_j) \to 0$ has a convergent subsequence. Indeed, every such sequence is bounded by coercivity and hence contains a convergent subsequence by a standard argument. First we show that $J_\eps$ has a minimizer $u^\eps_0$. Note that $J$ is also bounded from below. By $(g_2)$, there exists a $\lambda^\ast > 0$ such that for all $\lambda > \lambda^\ast$,
\begin{equation} \label{2.9}
c_1(\lambda) := \inf_{u \in H^1_0(\Omega)}\, J(u) < - \vol{\Omega}.
\end{equation}
For $\lambda > \lambda^\ast$, set
\[
\eps_0(\lambda) = \min \set{\frac{|c_1(\lambda)|}{2 \lambda a_1 \vol{\Omega}},\left(\frac{p a_1}{a_2}\right)^{1/(p-1)}}.
\]

\begin{lemma} \label{Lemma 2.3}
For all $\lambda > \lambda^\ast$ and $\eps < \eps_0(\lambda)$, $J_\eps$ has a minimizer $u^\eps_0 > 0$ satisfying
\begin{equation} \label{2.10}
J_\eps(u^\eps_0) \le c_1(\lambda) + 2 \lambda \eps a_1 \vol{\Omega} < 0.
\end{equation}
\end{lemma}

\begin{proof}
Since $J_\eps$ is bounded from below and satisfies the \PS{} condition, it has a minimizer $u^\eps_0$. Since $\B((t - 1)/\eps) \le \goodchi_{(1,\infty)}(t)$ for all $t$,
\begin{eqnarray*}
J_\eps(u) - J(u) & \le & \lambda \int_\Omega \big[G(x,(u - 1)_+) - G_\eps(x,(u - 1)_+)\big]\, dx\\[7.5pt]
& = & \lambda \int_\Omega \int_0^{(u - 1)_+} \left[1 - B\left(\frac{t}{\eps}\right)\right] g(x,t)\, dt\, dx\\[7.5pt]
& \le & \lambda \int_\Omega \int_0^\eps g(x,t)\, dt\, dx\\[7.5pt]
& \le & \lambda \left(a_1 \eps + \frac{a_2}{p}\, \eps^p\right) \vol{\Omega}
\end{eqnarray*}
by $(g_1)$, and \eqref{2.10} follows from this for $\eps < \eps_0(\lambda)$. Since $J_\eps(u^\eps_0) < 0 = J_\eps(0)$, $u^\eps_0$ is nontrivial and hence positive.
\end{proof}

Next we show that $J_\eps$ has a second nontrivial critical point $u^\eps_1$ using the mountain pass lemma of Ambrosetti and Rabinowitz \cite{MR0370183}, which we now recall.

\begin{lemma}[{\cite[Theorem 2.1]{MR0370183}}] \label{Lemma 2.4}
Let $I$ be a $C^1$-functional defined on a Banach space $X$. Assume that $I$ satisfies the {\em \PS{}} condition and that there exist an open set $U \subset X$, $u_0 \in U$, and $u_1 \in X \setminus \closure{U}$ such that
\[
\inf_{u \in \bdry{U}}\, I(u) > \max \set{I(u_0),I(u_1)}.
\]
Then $I$ has a critical point at the level
\[
c := \inf_{\gamma \in \Gamma}\, \max_{u \in \gamma([0,1])}\, I(u) \ge \inf_{u \in \bdry{U}}\, I(u),
\]
where $\Gamma = \bgset{\gamma \in C([0,1],X) : \gamma(0) = u_0,\, \gamma(1) = u_1}$ is the class of paths in $X$ joining $u_0$ and $u_1$.
\end{lemma}

\begin{lemma} \label{Lemma 2.5}
For all $\lambda > \lambda^\ast$, there exists a constant $c_2(\lambda) > 0$ such that for all $\eps < \eps_0(\lambda)$, $J_\eps$ has a second critical point $0 < u^\eps_1 \le u^\eps_0$ satisfying
\[
c_2(\lambda) \le J_\eps(u^\eps_1) \le \half \norm{u^\eps_0}^2 + \vol{\Omega}.
\]
In particular, $\set{u^\eps_0 > 1} \supset \set{u^\eps_1 > 1} \ne \emptyset$.
\end{lemma}

\begin{proof}
For $\eps < \eps_0(\lambda)$, let
\[
\begin{split}
\beta_\eps(x,s) = \frac{1}{\eps}\, \beta\left(\frac{\min \set{s,u^\eps_0(x)} - 1}{\eps}\right), \qquad \B_\eps(x,s) = \int_0^s \beta_\eps(x,t)\, dt,\\[7.5pt]
\widetilde{g}_\eps(x,s) = g_\eps(x,(\min \set{s,u^\eps_0(x)} - 1)_+), \qquad \widetilde{G}_\eps(x,s) = \int_0^s \widetilde{g}_\eps(x,t)\, dt
\end{split}
\]
and set
\[
\widetilde{J}_\eps(u) = \int_\Omega \left[\half\, |\nabla u|^2 + \B_\eps(x,u) - \lambda\, \widetilde{G}_\eps(x,u)\right] dx, \quad u \in H^1_0(\Omega).
\]
The functional $\widetilde{J}_\eps$ is of class $C^1$ and its critical points coincide with weak solutions of the problem
\[
\left\{\begin{aligned}
- \Delta u & = - \beta_\eps(x,u) + \lambda\, \widetilde{g}_\eps(x,u) && \text{in } \Omega\\[10pt]
u & = 0 && \text{on } \bdry{\Omega}.
\end{aligned}\right.
\]
If $u$ is a weak solution of this problem, then $u$ is also a classical solution by elliptic regularity theory and $u \le u^\eps_0$ by the maximum principle. So $u$ is a solution of problem \eqref{2.1}, and hence a critical point of $J_\eps$, with $J_\eps(u) = \widetilde{J}_\eps(u)$. We will show that $\widetilde{J}_\eps$ has a critical point $u^\eps_1$ satisfying
\[
c_2(\lambda) \le \widetilde{J}_\eps(u^\eps_1) \le \half \norm{u^\eps_0}^2 + \vol{\Omega}
\]
for some constant $c_2(\lambda) > 0$. This will prove the lemma since it follows from $J_\eps(u^\eps_1) = \widetilde{J}_\eps(u^\eps_1) > 0 > J_\eps(u^\eps_0)$ that $u^\eps_1$ is positive and distinct from $u^\eps_0$.

We apply Lemma \ref{Lemma 2.4} to the functional $\widetilde{J}_\eps$, which is also coercive and hence satisfies the \PS{} condition. Since $\widetilde{g}_\eps(x,s) = g_\eps(x,0) = 0$ for $s \le 1$ and
\[
\widetilde{g}_\eps(x,s) \le a_1 + a_2\, (\min \set{s,u^\eps_0(x)} - 1)_+^{p-1} \le a_1 + a_2\, (s - 1)^{p-1}
\]
for $s > 1$ by $(g_1)$,
\[
\widetilde{G}_\eps(x,s) \le a_1\, (s - 1)_+ + \frac{a_2}{p}\, (s - 1)_+^p \le \left(a_1 + \frac{a_2}{p}\right) |s|^q
\]
for all $s$, where $q > 2$ if $N = 2$ and $2 < q \le 2N/(N - 2)$ if $N \ge 3$. Since $\B_\eps(x,s) \ge 0$ for all $s$, then
\[
\widetilde{J}_\eps(u) \ge \int_\Omega \left[\half\, |\nabla u|^2 - \lambda \left(a_1 + \frac{a_2}{p}\right) |u|^q\right] dx.
\]
Since $L^q(\Omega) \hookrightarrow H^1_0(\Omega)$ and $q > 2$, the infimum $c_2(\lambda)$ of the last integral on $\bdry{B_\rho(0)}$ is positive for all sufficiently small $\rho > 0$, where $B_\rho(0) = \set{u \in H^1_0(\Omega) : \norm{u} < \rho}$. Since $\widetilde{J}_\eps(u^\eps_0) = J_\eps(u^\eps_0) < 0 = \widetilde{J}_\eps(0)$, taking $\rho < \norm{u^\eps_0}$ and applying Lemma \ref{Lemma 2.4} now gives a critical point $u^\eps_1$ of $\widetilde{J}_\eps$ with
\[
\widetilde{J}_\eps(u^\eps_1) = \inf_{\gamma \in \Gamma}\, \max_{u \in \gamma([0,1])}\, \widetilde{J}_\eps(u) \ge \inf_{u \in \bdry{B_\rho(0)}}\, \widetilde{J}_\eps(u) \ge c_2(\lambda),
\]
where $\Gamma = \set{\gamma \in C([0,1],H^1_0(\Omega)) : \gamma(0) = 0,\, \gamma(1) = u^\eps_0}$ is the class of paths joining $0$ and $u^\eps_0$. For the path $\gamma_0(t) = tu^\eps_0,\, t \in [0,1]$,
\[
\widetilde{J}_\eps(\gamma_0(t)) \le \int_\Omega \left(\half\, |\nabla u^\eps_0|^2 + \B_\eps(x,u^\eps_0)\right) dx
\]
since $\B_\eps(x,s)$ is nondecreasing in $s$ and $\widetilde{G}_\eps(x,s) \ge 0$ for all $s$ by $(g_2)$. Since
\[
\B_\eps(x,u^\eps_0(x)) = \int_0^{u^\eps_0(x)} \frac{1}{\eps}\, \beta\left(\frac{t - 1}{\eps}\right) dt = \B\left(\frac{u^\eps_0(x) - 1}{\eps}\right) \le 1,
\]
then
\[
\widetilde{J}_\eps(u^\eps_1) \le \max_{u \in \gamma_0([0,1])}\, \widetilde{J}_\eps(u) \le \int_\Omega \left(\half\, |\nabla u^\eps_0|^2 + 1\right) dx = \half \norm{u^\eps_0}^2 + \vol{\Omega}. \QED
\]
\end{proof}

We are now ready to prove Theorem \ref{Theorem 1.1}.

\begin{proof}[Proof of Theorem \ref{Theorem 1.1}]
Let $\lambda > \lambda^\ast$ and take a sequence $\eps_j \searrow 0$ with $\eps_j < \eps_0(\lambda)$. For each $j$, Lemma \ref{Lemma 2.3} gives a minimizer $u^{\eps_j}_0 > 0$ of $J_{\eps_j}$ satisfying
\begin{equation} \label{2.11}
J_{\eps_j}(u^{\eps_j}_0) \le c_1(\lambda) + 2 \lambda \eps_j\, a_1 \vol{\Omega} < 0
\end{equation}
and Lemma \ref{Lemma 2.5} gives a second critical point $0 < u^{\eps_j}_1 \le u^{\eps_j}_0$ satisfying
\begin{equation} \label{2.12}
c_2(\lambda) \le J_{\eps_j}(u^{\eps_j}_1) \le \half \norm{u^{\eps_j}_0}^2 + \vol{\Omega}.
\end{equation}
We will show that the sequences $\seq{u^{\eps_j}_0}$ and $\seq{u^{\eps_j}_1}$ are bounded in $H^1_0(\Omega) \cap L^\infty(\Omega)$ and apply Lemma \ref{Lemma 2.1}.

Since $\B \ge 0$ and
\[
G_\eps(x,(s - 1)_+) \le a_1\, (s - 1)_+ + \frac{a_2}{p}\, (s - 1)_+^p \le \left(a_1 + \frac{a_2}{p}\right) |s|^p
\]
for all $s$ by $(g_1)$,
\[
\half \norm{u^\eps_0}^2 \le J_\eps(u^\eps_0) + \lambda \left(a_1 + \frac{a_2}{p}\right) \int_\Omega (u^\eps_0)^p\, dx.
\]
Since $J_{\eps_j}(u^{\eps_j}_0) < 0$ by \eqref{2.11} and $p < 2$, it follows from this that $\seq{u^{\eps_j}_0}$ is bounded in $H^1_0(\Omega)$. Then $J_{\eps_j}(u^{\eps_j}_1)$ is bounded by \eqref{2.12}, so a similar argument shows that $\seq{u^{\eps_j}_1}$ is also bounded in $H^1_0(\Omega)$.

Since $g_\eps(x,(s - 1)_+) = g_\eps(x,0) = 0$ for $s \le 1$ and
\[
g_\eps(x,(s - 1)_+) \le a_1 + a_2\, (s - 1)^{p-1} \le (a_1 + a_2)\, s^{p-1}
\]
for $s > 1$ by $(g_1)$,
\[
- \Delta u^{\eps_j}_0 = - \frac{1}{\eps_j}\, \beta\left(\frac{u^{\eps_j}_0 - 1}{\eps_j}\right) + \lambda\, g_{\eps_j}(x,(u^{\eps_j}_0 - 1)_+) \le \lambda\, (a_1 + a_2)\, (u^{\eps_j}_0)^{p-1}.
\]
This together with the fact that $\seq{u^{\eps_j}_0}$ is bounded in $H^1_0(\Omega)$ implies that $\seq{u^{\eps_j}_0}$ is also bounded in $L^\infty(\Omega)$ (see, e.g., Bonforte et al.\! \cite[Theorem 3.1]{MR3155966}). Then so is $\seq{u^{\eps_j}_1}$ since $0 < u^{\eps_j}_1 \le u^{\eps_j}_0$.

By Lemma \ref{Lemma 2.1}, for a renamed subsequence of $\seq{\eps_j}$, the sequences $\seq{u^{\eps_j}_0}$ and $\seq{u^{\eps_j}_1}$ converge uniformly to Lipschitz continuous solutions $u_0, u_1 \in H^1_0(\Omega) \cap C^2(\closure{\Omega} \setminus F(u))$ of problem \eqref{1.1} that satisfy the equation $- \Delta u = \lambda\, \goodchi_{\set{u > 1}}(x)\, g(x,(u - 1)_+)$ in the classical sense in $\Omega \setminus F(u)$, the free boundary condition in the generalized sense, and vanish continuously on $\bdry{\Omega}$. Moreover,
\begin{equation} \label{2.13}
J(u_0) \le \liminf J_{\eps_j}(u^{\eps_j}_0) \le \limsup J_{\eps_j}(u^{\eps_j}_0) \le J(u_0) + \vol{{\set{u_0 = 1}}}
\end{equation}
and
\begin{equation} \label{2.14}
J(u_1) \le \liminf J_{\eps_j}(u^{\eps_j}_1) \le \limsup J_{\eps_j}(u^{\eps_j}_1) \le J(u_1) + \vol{{\set{u_1 = 1}}}.
\end{equation}
Combining \eqref{2.13} with \eqref{2.11} and \eqref{2.9} gives $J(u_0) \le \limsup J_{\eps_j}(u^{\eps_j}_0) \le c_1(\lambda) \le J(u_0)$, so
\begin{equation} \label{2.15}
J(u_0) = c_1(\lambda) < - \vol{\Omega}.
\end{equation}
On the other hand, combining \eqref{2.14} with \eqref{2.12} gives $J(u_1) + \vol{{\set{u_1 = 1}}} \ge \liminf J_{\eps_j}(u^{\eps_j}_1) \ge c_2(\lambda) > 0$, so
\begin{equation} \label{2.16}
J(u_1) > - \vol{{\set{u_1 = 1}}} \ge - \vol{\Omega}.
\end{equation}
It follows from \eqref{2.15} and \eqref{2.16} that $u_0$ and $u_1$ are nontrivial and distinct, $u_0$ is a minimizer of $J$, and $u_1$ is not a minimizer. Since $u^{\eps_j}_1 \le u^{\eps_j}_0$ for all $j$, $u_1 \le u_0$. Since $u_1$ is nontrivial, then $0 < u_1 \le u_0$, the sets $\set{u_0 < 1} \subset \set{u_1 < 1}$ are connected if $\bdry{\Omega}$ is connected, and the sets $\set{u_0 > 1} \supset \set{u_1 > 1}$ are nonempty.
\end{proof}

\def\cdprime{$''$}

\end{document}